\documentclass[final]{article}
\usepackage{multicol} 			
\usepackage{dsfont}
\usepackage[left=2cm,right=2cm,bottom=2cm,top=2cm]{geometry}
\usepackage{caption}
\usepackage{floatrow}
\usepackage{graphicx}

\usepackage{amsmath,amssymb,amsthm}
\usepackage[dvips]{epsfig}

\usepackage{geometry}

\usepackage{colortbl}
\usepackage{xcolor}
\usepackage {tikz}
\usepackage{subfigure}

\newtheorem{theo}{Theorem}[section]
\newtheorem{lem}[theo]{Lemma}
\newtheorem{cor}[theo]{Corollary}

\usepackage[utf8]{inputenc}
\usepackage{amsmath}
\usepackage{amsfonts}
\usepackage{url}

\newtheorem{definition}{Definition}[section]

\usepackage{lipsum}
\usepackage{amsthm}

\makeatletter							
\renewcommand\@maketitle{%
\null									
{
\large					
\@title \par}%
\vskip 0.6em%
{
\large						
\lineskip .5em%
\begin{tabular}[t]{l}%
\@author
\end{tabular}\par}%
\vskip 1cm
\par
}
\makeatother

\title{Symmetric Chromatic Polynomial of Trees}

\author{Isaac Smith, Zane Smith, \& Peter Tian\\
The Ohio State University}

\begin{document}

\begin{minipage}{\textwidth}					
\maketitle
\end{minipage}
\vspace{1cm}

\raggedcolumns							


\large{

\begin{abstract}

In a 1995 paper Richard Stanley defined $X_G$, the symmetric chromatic polynomial of a Graph $G=(V,E)$. He then conjectured that $X_G$ distinguishes trees; a conjecture which still remains open. $X_G$ can be represented as a certain collection of integer partitions of $|V|$ induced by each $S\subseteq E$, which is very approachable with the aid of a computer. Our research involved writing a computer program for efficient verification of this conjecture for trees up to 23 vertices. In this process, we also gather trees with matching collections of integer partitions of a fixed number of parts. For each $k=2, 3, 4, 5$, we provide the smallest pair of trees whose partitions of $k$ parts agree. In 2013, Orellana and Scott give a proof of a weaker version of Stanely's conjecture for trees with one centroid. We prove a similar result for arbitrary trees, and provide examples to show that this result, combined with that of Orellana and Scott, is optimal.

\end{abstract}

\maketitle

\section{Introduction}

Let $G = (V,E)$ be a finite graph, $V = \{v_1,v_2,...,v_n\}$ the set of vertices, and $E = \{e_1,e_2,...,e_{m}\}$ the set of edges.
The Chromatic Symmetric Polynomial is a function of countably many commuting indeterminates, and is defined as:
    \begin{equation} 
        X_G(x_1,x_2,...) = \sum_{\kappa} \prod_{i=1}^{n} x_{\kappa(i)}
    \end{equation}
Where $\kappa:\{1,2,...,n\} \rightarrow \mathds{N}$ is a coloring of G and the sum above is taken over all \textit{proper} colorings, i.e. colorings in which no edge connects two vertices of the same color.
In \cite{St}, Stanley provides another form of $X_G$ that we will primarily be using throughout our research. The power-sum functions, $p_m := \sum_{i=1}^{\infty} x_{i}^{m}$ 
provide a basis for the space of symmetric polynomials and $X_G$ can be written as:
    \begin{equation}
    \label{sch-in-p}
        X_G = \sum_{S \subseteq E} (-1)^{n -1-\#S}p_{\theta_{G}(S)}
    \end{equation}
Where $\theta_{G}(S) = (t_1,t_2,...,t_r)$, is an integer partition of $n$ with each $t_j$ corresponding to the number of vertices in a component of $G$ after removing every edge of $S$.
And where $p_{\theta_{G}(S)} := p_{t_1}p_{t_2}...p_{t_r}$ is a product of power sum functions. 
 Also in \cite{St}, Stanley both gives the definition above, and the proposes the following conjecture: 
\\*\newtheorem*{theorem}{Conjecture}
\begin{theorem}
 \hspace{20pt}$X_G$ distinguishes trees
\end{theorem}

That is, for any two non-isomorphic trees, $G_1$ and $G_2$, we have $X_{G_1} \neq X_{G_2}$.
\\*
\\*It's already known that $X_G$ does not distinguish graphs in general (see figure below).

\pagebreak[4]

\centerline{\includegraphics[width=.4\textwidth]{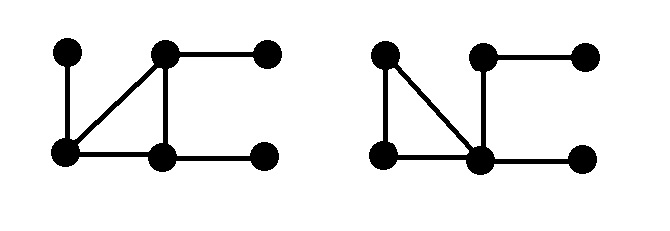}}
Two non-isomorphic graphs with the same Symmetric Chromatic Polynomial. From now on all graphs mentioned will be assumed to be trees, i.e. connected graphs without cycles.
\\*
\\*In \cite{APZ}, Aliste-Prieto and Zamora proved this conjecture for a special subset of trees known as ``Caterpillars''. Caterpillars are trees consisting of a path (the ``spine'') and each vertex of the spine having an arbitrary number of leaves (the ``legs'').
\\*
\\*The representation given in \eqref{sch-in-p} is much easier for the computation of $X_G$, whether by hand or computer. Using this representation, we identify $X_G$ with the the multiset of partitions formed by taking subsets of edges as described above. The main focus of our research was to write a program to search through all trees up to 23 vertices and record the specific ones which have many matching terms of $X_G$. The number 23 was chosen due to hardware limits. Since the number of unique trees grows approximately as $3^n$, each increment in size of trees we consider triples all computation times.
\\*

\pagebreak[2]

For a concrete example of the computations we're performing, consider the two non-isomorphic trees with 4 vertices:
\\*
\centerline{\includegraphics[width=.4\textwidth]{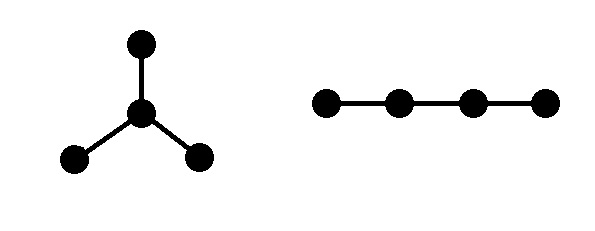}}
\\*
Let $T_1$ be the tree on the left. We then calculate $X_{T_1}$ using \eqref{sch-in-p}. There are a total of 8 terms, corresponding to the 8 subsets of our 3-element edge set. We will refer to the edges as 1, 2, and 3. Then we find $\theta_{T_1}(S)$ for each $S \subseteq \{1,2,3\}$.

\begin{center}
$\theta_{T_1}(\{1,2,3\}) = (1,1,1,1), \; \; \theta_{T_1}(\{1,2\})=\theta_{T_1}(\{2,3\})=\theta_{T_1}(\{1,3\}) = (1,1,2)$
\\*
$\theta_{T_1}(\{1\})=\theta_{T_1}(\{2\})=\theta_{T_1}(\{3\}) = (1,3), \; \;$and$\; \; \theta_{T_1}(\{\})=(4)$
\\*
\end{center}

So the partitions given by $T_1$ are $(1,1,1,1)$; $3(1,1,2)$; $3(1,3)$; and $(4)$. Therefore by equation \eqref{sch-in-p}, we have:

\begin{center}
 $X_{T_1} = p_{(1,1,1,1)} - 3p_{(1,1,2)} + 3p_{(1,3)} - p_{(4)} = {p_1}^4 -3{p_1}^2{p_2} + 3p_1p_3 - p_4$ 
 \end{center}
 \begin{center}
 $= (x_1 + x_2 + ...)^4 -3(x_1+x_2+...)^2(x_1^2 + x_2^2 +...) +3(x_1+x_2+...)(x_1^3 +x_2^3+...) -(x_1^4 +x_2^4+...)$
\end{center}
 Call the tree (path) on the right $T_2$. Then after doing the same, we have: 
 \begin{center}
    $X_{T_2} = p_{(1,1,1,1)} - 3p_{(1,1,2)} + 2p_{(1,3)} + p_{(2,2)} - p_{(4)} = {p_1}^4 -3{p_1}^2{p_2} + 2p_1p_3 + p_2^2- p_4$
 \end{center}
 
As we see, $X_{T_1} \neq X_{T_2}$. Specifically, $T_2$ has an edge corresponding to the partition $(2,2)$, while $T_1$ does not. It suffices to consider the multi-set of partitions, rather than the full symmetric polynomial representation. Indeed, the Fundamental Theorem of Symmetric Polynomials says that any symmetric polynomial has a unique representation with elementary symmetric polynomials. And since power sum functions can be uniquely reduced to elementary symmetric polynomials by Newton's identities, we see that the multi-set of partitions described in representation \eqref{sch-in-p} uniquely determines $X_G$.
 \\*
 \\*Considering subsets formed by removing all but one edge, one receives practically no information about a tree. There are $n-1$ such subsets, since all trees on $n$ vertices have $n-1$ edges, and they all give the partition (2,1,1,...,1). 
In contrast, there are also $n-1$ subsets formed by removing a single edge, but they give much more information about the tree. For example, the number of partitions of the form $(1,n-1)$ is exactly the number of leaves in the tree, also the partition $(n/2,n/2)$ appears iff the tree has two centroids.
\\*
\\*Subsets of this form can each be thought of as removing a single edge from G, hence we refer to such a subset as a ``1-cut'' and for convenience we refer to the set $\{\theta_{G}(S): S \subseteq E $ and $ |S|=1\}$ as the ``1-cuts'' of a given graph. In general, the "k-cuts" of a graph is the set $\{\theta_{G}(S): S \subseteq E $ and $ |S|=k\}$. In the example given above, we see $T_1$ and $T_2$ have identical 2-cuts, but have different 1-cuts. 
\\*

\pagebreak[4]
\*Stanley's conjecture can be weakened by considering ``labeled'' cuts. Usually we are just considering the multiset $\{\theta_{T}(S): S \subseteq E\}$. In this way we identify the terms of \eqref{sch-in-p} with elements of this multiset. However, we could consider a stronger form of this multiset, in which each element includes the subset of edges that were cut to create it. That is, the set of ``labeled cuts'' is $\{(S,\theta_{T}(S)):S \subseteq E\}$. In \cite{OS}, Orellana and Scott prove that ``labeled 2-cuts'', i.e. $\{(S,\theta_{T}(S)):S \subseteq E, |S|=2\}$ is enough to distinguish trees with single centroids. We optimize this result by both providing the smallest two trees (with two centroids) which have identical labeled 2-cuts, then proving that for any $3\leq k \leq n-3$, labeled k-cuts are enough to distinguish any tree. These proofs involving labeled k-cuts do not extend to the general case of unlabeled cuts since they heavily rely on identifying an edge with its 1-cut then examining all larger cuts that include the same edge.
{\bf Remark.} The representation of $X_G$ given by \eqref{sch-in-p} was discovered independently (in a polynomial form) in \cite{CDL}. Their motivation comes from knot theory and uses the Hopf algebra structure on the space spanned by forests. The polynomial defined there, the {\it weighted chromatic polynomial}, is defined for graphs with integer weights at vertices. As such it satisfies a deletion/contraction relation: it's value on a weighted graph is equal to the sum of it's values on a graph with an edge deleted plus its values on a graph with the edge contracted. 

\begin{equation} 
        \label{wcp}
        X_G(x_1,x_2,...) = X_{G-e} + X_{G\backslash e}
\end{equation}

Where contracting an edge creates a vertex with weight having the sum of the weights of the two vertices previously connected by that edge. Finally we have the recursive base case when $G$ has no edges:

\begin{equation} 
        X_G(x_1,x_2,...) = \prod_{i=1}^{n} x_{w(i)}
    \end{equation}

Where $w(i)$ is the weight of each isolated vertex. The equivalence of \eqref{wcp} to Stanley's symmetric chromatic function was established in \cite{NW}. In light of the definition above, the reader can see the equivalence (up to a negative sign) by corresponding each term of \eqref{sch-in-p} to the recursive step in which all the edges of $S$ are removed, and each connected component is contracted to 1 vertex. 
    In \cite{NW}, Noble and Welsh generalized the weighted chromatic polynomial above to all graphs. Similarly to the Tutte polynomial, they add a new variable and base case for edges which are ``loops'' (connected to the same vertex on both ends). However, for trees and forests this variable does not occur and their polynomial is reduced to the original weighted chromatic polynomial of \cite{CDL}.

\pagebreak[4]

\maketitle

\section{Constructing any Tree from Labeled 3-cuts}

In this section, we prove a generalization of a classification theorem given by Theorem 5.7 in Orellana and Scott. Their theorem states that \textit{labeled} 2-cuts distinguish trees with only one centroid. We extend this result by showing that for any $3 \le k \le n-3$, labeled $k$-cuts distinguish all trees, whether they have one centroid or two. This result is significant because it may give us insight into the problem of showing that unlabeled $k$-cuts, for some $k$, distinguishes trees, which implies Stanley's Conjecture.

We use the same $\theta_{G}(S)$ as defined above, e.g. if $T$ is the left tree in Figure 1 below, $\theta_T(\{e_2, e_3\})=(8,3,3)$. Note that we are assuming knowledge of the \textit{map} $\theta_{G}$, so for each subset of edges, we know the corresponding partition of $n$. This is what we mean by \textit{labeled} cuts, as opposed to the unlabeled case where we only know the multiset of partitions generated by subsets of $E$. Stanley's conjecture would be implied by results of this kind for unlabeled cuts, but not necessarily the results we actually prove for labeled cuts. For consistency, we order partitions in reverse lexicographical order: Given $n_1 \ge \ldots \ge n_\ell$ and $m_1 \ge \ldots \ge m_\ell$, $(n_1, \ldots, n_\ell)>(m_1, \ldots, m_\ell)$ if for the first index $i$ such that $n_i \not =m_i$, we have $n_i <m_i$.

The weight of a vertex in a tree is the maximal number of edges in any subtree that contains the vertex as a leaf. We call a vertex of $T$ a {\it centroid} if it has minimum weight. It is known that trees always have one or two centroids. In Figures 1 and 2, $c_1$ and $c_2$ are the centroids of the trees.

It is not hard to show that there exists an edge $e$ adjacent to a centroid such that $\theta_T(\{e\})$ is maximal among all edges $e$ in $T$. Additionally, if a tree has two centroids, then the edge connecting the two centroids is the unique edge $e$ such that $\theta_T(\{e\})=({n \over 2}, {n \over 2})$. The following lemma is another result about double-centroid trees.

\begin{lem}
\label{3-cuts} 
If $T$ is a double centroid tree on $n \ge 4$ vertices, $e_1$ is the edge connecting its two centroids, and $e_i \not = e_j$ are any other edges, then $\theta_T(\{e_1,e_i,e_j\})$ is of the form $({n \over 2}, n_1, n_2, n_3)$, where $n_i<{n \over 2},$ if $e_i$ and $e_j$ are in the same connected component after $e_1$ is removed and $\theta_T(\{e_1,e_2,e_i\})$ is of the form $(m_1, m_2, m_3, m_4)$, where $m_i < {n \over 2}$, if $e_i$ and $e_j$ are in different connected components after $e_1$ is removed. 
\end{lem}
\begin{proof}
The proof is straightforward.
\end{proof}

The proof of the following theorem is similar to the proof of Theorem 5.7 in \cite{OS}. The key difference in our approach is using
Lemma \ref{3-cuts}.

\begin{theo}
\label{2c}
Let $T$ be a double-centroid tree on $n \ge 4$ vertices with edge set $\{e_1, \ldots, e_{n-1}\}$, where $\theta_T(\{e_1 \}) \ge \cdots \ge \theta_T(\{e_{n-1}\})$. Then the data $\theta_T(\{e_i,e_j\})$ and $\theta_T(\{e_1,e_2,e_k\})$ for all edges $e_i,e_j,e_k$, where $e_i \not=e_j$, uniquely determines $T$.
\end{theo}
\begin{proof}

First note that by the same proof as Theorem 5.8 in \cite{OS}, we can indeed determine $\theta_T(\{e_i\})$, for any edge $e_i$, using the data $\theta_T(\{e_i,e_j\})$, for any edges $e_i \not =e_j$. Later we will fully extend this statement by showing that for any $3 \leq k \leq n-3$, we can recover the set of labeled $k-1$-cuts from the set of labeled $k$-cuts. (Orrelana and Scott prove this for $k=2$.) 

As in \cite{OS}, we construct $T$ by adding edges in the order of $e_1, \ldots, e_{n-1}$. We need to show that each edge $e_i$ must be connected to a unique vertex of the graph $T_{i-1}$ consisting of the edges $e_1, \ldots, e_{i-1}$. 

So we will prove that we can uniquely determine $T_i$ for all $i\le n-1$ by induction. The base case is $i=1$ and $i=2$. $i=1$ is trivial, since $T_1$ must be the single edge $e_1$, which must connect the two centroids. Additionally, $e_2$ must be adjacent to $e_1$ or else there is an edge $e_j$, where $j>2$, separating $e_2$ from $e_1$, implying that $\theta(\{e_j\})>\theta(\{e_2\})$, a contradiction. Without loss of generality, we assume that $e_2$ is attached to $c_1$. Now assume that we can uniquely determine $T_{i-1}$ for some $i\ge 2$. We show that $e_i$ must be attached to a unique vertex in $T_{i-1}$, thus showing that $T_i$ is uniquely determined.

Note that $e_i$ must be connected to $T_{i-1}$, for otherwise there is an edge $e_j$, where $j>i$, separating $e_i$ from $e_1$, implying that $\theta_T(\{e_j\})>\theta(\{e_i\})$, a contradiction.

Before we proceed we define the notion of attraction and repulsion, which were introduced by Orellana and Scott. We say that two edges $e_1 \not = e_2$ of $T$ {\it attract} if there is a path in $T$ that begins at a centroid and contains both $e_1$ and $e_2$. If $e_1$ and $e_2$ do not attract, they are said to {\it repel}. Note that if a graph has two centroids, then the edge connecting the two centroids attracts every other edge in the tree. Orellana and Scott's Proposition 5.2 says that if $\theta_T(\{e_1\})=(n-i,i)$ and $\theta_T(\{e_2\})=(n-k,k)$, where ${n \over 2} \ge i \ge k$, then $\theta_T(\{ e_1, e_2\})=(n-i-k,i,k)$ if $e_1$ and $e_2$ attract and $\theta_T(\{ e_1, e_2\})=(n-i, i-k,k)$ if $e_1$ and $e_2$ repel. Thus since we know $\theta_T(e_i)$, for any edge $e_i$, and  $\theta_T(\{e_i,e_j\})$, for any edges $e_i \not =e_j$, we can determine whether any two edges $e_i \not =e_j$ attract or repel.

Now we can do casework based on which edges $e_1, \ldots, e_{i-1}$ attract $e_i$. Case 1 is Orellana and Scott's idea but
Case 2 is the key difference between our proof and Orellana and Scott's.

\noindent
{\bf Case 1: More than one edge in $T_i$ attracts $e_i$.} 

Then these edges form a path starting with $e_1$ and ending at a vertex $v \not = c_1, c_2$ of $T_{i-1}$. Since $e_1$ is connected to $T_{i-1}$, it must be attached to $v$. 

\noindent
{\bf Case 2: Only $e_1$ attracts $e_i$ }

Then $e_i$ must be attached to $c_1$ or $c_2$. By Lemma \ref{3-cuts} we can use $\theta_T(\{e_1,e_2,e_i\})$ to determine whether $e_i$ is on the same side of $e_1$ as $e_2$ is. This in turn tells us whether $e_i$ is attached to $c_1$ or $c_2$.

In either case we have shown that there is a unique vertex in $T_{i-1}$ to which $e_i$ is attached. Thus $T_i$ is a uniquely determined tree, so our induction is complete.
\end{proof}

Orellana and Scott gave an example of two trees with two centroids that have exactly the same values of $\theta_T(\{e_i,e_j\})$. They are shown in Figure 1 below.

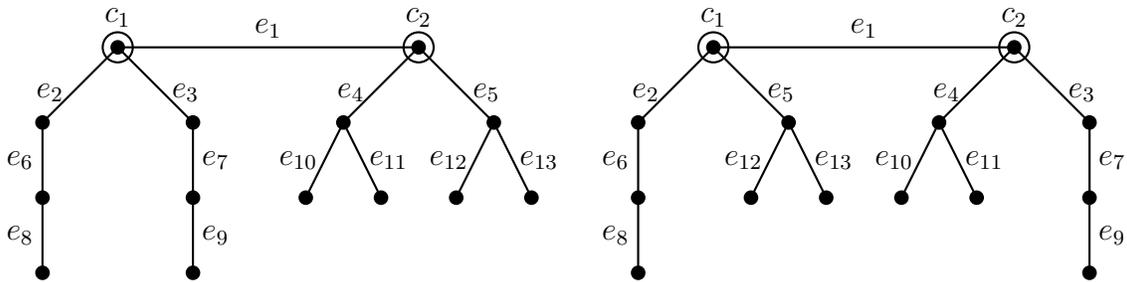
\begin{figure}[htbp]
\subfigure{
\begin {tikzpicture}[ xscale=1.0,yscale=1.0, thick,font=\large]

\tikzset{My Style/.style={draw,circle, scale=0.5, fill=black}}

 \node[My Style] (1)  at (0,0)  {};
 \node[My Style] (2)  at (4,0)  {};
 \node[My Style] (3)  at (-1,-1)  {};
 \node[My Style] (4)  at (-1,-2)  {};
 \node[My Style] (5)  at (-1,-3)  {};
 \node[My Style] (6)  at (1,-1)  {};
 \node[My Style] (7)  at (1,-2)  {};
 \node[My Style] (8)  at (1,-3)  {};
 \node[My Style] (9)  at (3,-1)  {};
 \node[My Style] (10)  at (2.5,-2)  {};
 \node[My Style] (11)  at (3.5,-2)  {};
 \node[My Style] (12)  at (5,-1)  {};
 \node[My Style] (13)  at (4.5,-2)  {};
 \node[My Style] (14)  at (5.5,-2)  {};

\draw (0,0) circle (0.2);
\draw (4,0) circle (0.2);

\draw (1) edge (2); 
\draw (1) edge (3);
\draw (3) edge (4);
\draw (4) edge (5);
\draw (1) edge (6);
\draw (6) edge (7);
\draw (7) edge (8);
\draw (2) edge (9);
\draw (9) edge (10);
\draw (9) edge (11);
\draw (2) edge (12);
\draw (12) edge (13);
\draw (12) edge (14);

 \node at (2,0.25)  {$e_1$};
 \node at (-0.9,-0.6)  {$e_2$};
 \node at (0.9,-0.6)  {$e_3$};
 \node at (-1.3,-1.5)  {$e_6$};
 \node at (-1.3,-2.5)  {$e_8$};
 \node at (1.3,-1.5)  {$e_7$};
 \node at (1.3,-2.5)  {$e_9$};

 \node at (3.1,-0.6)  {$e_4$};
 \node at (4.9,-0.6)  {$e_5$};
 \node at (2.4,-1.5)  {$e_{10}$};
 \node at (3.6,-1.5)  {$e_{11}$};
 \node at (4.4,-1.5)  {$e_{12}$};
 \node at (5.6,-1.5)  {$e_{13}$};

 \node at (0,0.4)  {$c_1$};
 \node at (4,0.4)  {$c_2$};

\end{tikzpicture}
}
\subfigure{
\begin {tikzpicture}[ xscale=1.0,yscale=1.0, thick,font=\large]

\tikzset{My Style/.style={draw,circle, scale=0.5, fill=black}}

 \node[My Style] (1)  at (0,0)  {};
 \node[My Style] (2)  at (4,0)  {};
 \node[My Style] (3)  at (-1,-1)  {};
 \node[My Style] (4)  at (-1,-2)  {};
 \node[My Style] (5)  at (-1,-3)  {};
 \node[My Style] (6)  at (1,-1)  {};
 \node[My Style] (7)  at (0.5,-2)  {};
 \node[My Style] (8)  at (1.5,-2)  {};
 \node[My Style] (9)  at (3,-1)  {};
 \node[My Style] (10)  at (2.5,-2)  {};
 \node[My Style] (11)  at (3.5,-2)  {};
 \node[My Style] (12)  at (5,-1)  {};
 \node[My Style] (13)  at (5,-2)  {};
 \node[My Style] (14)  at (5,-3)  {};

\draw (0,0) circle (0.2);
\draw (4,0) circle (0.2);

\draw (1) edge (2); 
\draw (1) edge (3);
\draw (3) edge (4);
\draw (4) edge (5);
\draw (1) edge (6);
\draw (6) edge (7);
\draw (6) edge (8);
\draw (2) edge (9);
\draw (9) edge (10);
\draw (9) edge (11);
\draw (2) edge (12);
\draw (12) edge (13);
\draw (12) edge (14);

 \node at (2,0.25)  {$e_1$};
 \node at (-0.9,-0.6)  {$e_2$};
 \node at (0.9,-0.6)  {$e_5$};
 \node at (-1.3,-1.5)  {$e_6$};
 \node at (-1.3,-2.5)  {$e_8$};
 \node at (5.3,-1.5)  {$e_7$};
 \node at (5.3,-2.5)  {$e_9$};

 \node at (3.1,-0.6)  {$e_4$};
 \node at (4.9,-0.6)  {$e_3$};
 \node at (2.4,-1.5)  {$e_{10}$};
 \node at (3.6,-1.5)  {$e_{11}$};
 \node at (0.4,-1.5)  {$e_{12}$};
 \node at (1.6,-1.5)  {$e_{13}$};

 \node at (0,0.4)  {$c_1$};
 \node at (4,0.4)  {$c_2$};

\end{tikzpicture}

}
\caption{Two non-isomorphic double-centroid trees with the same values of $\theta_T(e_i, e_j)$}
\end{figure}

Note that for the left tree $\theta_T(\{e_1,e_2,e_3\})=(7,3,3,1)$ and for the right tree $\theta_T(\{e_1,e_2,e_3\})=(4,4,3,3)$ . To fully determine either tree we actually just need the data $\theta_T(\{e_1,e_2,e_3\})$, $\theta_T(\{e_1,e_2,e_4\})$, $\theta_T(\{e_1,e_2,e_5\})$.

In the course of our research, we discovered the trees shown in Figure 2. They are the pair of trees with the fewest number of vertices that have identical values of $\theta_T(e_i, e_j)$. These trees have essentially only one way to label the edges, hence we are certain that labeled two-cuts are insufficient to distinguish them.

\begin{figure}[htbp]
\subfigure{
\begin {tikzpicture}[ xscale=1.2,yscale=1.0, thick,font=\large]

\tikzset{My Style/.style={draw,circle, scale=0.7, fill=black}}

 \node[My Style] (1)  at (1,0)  {};
 \node[My Style] (2)  at (-1,1.5)  {};
 \node[My Style] (3)  at (-1,0.5)  {};
 \node[My Style] (4)  at (-1,-0.5)  {};
 \node[My Style] (5)  at (-1,-1.5)  {};
 \node[My Style] (6)  at (4,1)  {};
 \node[My Style] (7)  at (4,-1)  {};
 \node[My Style] (8)  at (5,1.5)  {};
 \node[My Style] (9)  at (5,-1.5)  {};
 \node[My Style] (10)  at (3,0)  {};

\draw (1) edge (2); 
\draw (1) edge (3); 
\draw (1) edge (4); 
\draw (1) edge (5); 
\draw (10) edge (6); 
\draw (10) edge (7); 
\draw (6) edge (8); 
\draw (7) edge (9); 
\draw (10) edge (1);

\draw (1,0) circle (0.2);
\draw (3,0) circle (0.2);

 \node at (2,0.25)  {$e_1$};
 \node at (0,1)  {$e_4$};
 \node at (0,0.4)  {$e_5$};
 \node at (0,-0.5)  {$e_6$};
 \node at (0,-1.1)  {$e_7$};

 \node at (3.5,0.8)  {$e_2$};
 \node at (3.5,-0.8)  {$e_3$};

 \node at (4.5,1.5)  {$e_8$};
 \node at (4.5,-1.5)  {$e_9$};

 \node at (1,0.4)  {$c_1$};
 \node at (3,0.4)  {$c_2$};

 \node at (4,0)  {};

\end{tikzpicture}
}
\subfigure{
\begin {tikzpicture}[ xscale=1.2,yscale=1.0, thick,font=\large]

\tikzset{My Style/.style={draw,circle, scale=0.7, fill=black}}

 \node[My Style] (10)  at (1,0)  {};

 \node[My Style] (1)  at (0,1)  {};
 \node[My Style] (2)  at (-1,1.5)  {};

 \node[My Style] (3)  at (-1,-0.5)  {};
 \node[My Style] (4)  at (-1,-1.5)  {};

 \node[My Style] (5)  at (3,0)  {};

 \node[My Style] (6)  at (5,-0.5)  {};
 \node[My Style] (7)  at (5,-1.5)  {};

 \node[My Style] (8)  at (4,1)  {};
 \node[My Style] (9)  at (5,1.5)  {};

\draw (1,0) circle (0.2);
\draw (3,0) circle (0.2);

\draw (10) edge (3); 
\draw (10) edge (4); 
\draw (10) edge (1); 
\draw (1) edge (2); 
\draw (10) edge (5); 
\draw (5) edge (6); 
\draw (5) edge (7); 
\draw (5) edge (8); 
\draw (8) edge (9); 

 \node at (2,0.25)  {$e_1$};
 \node at (3.5,0.8)  {$e_2$};
 \node at (4.5,1.5)  {$e_8$};
 \node at (4,-0.4)  {$e_4$};
 \node at (4,-1)  {$e_5$};

 \node at (0.5,0.8)  {$e_3$};
 \node at (-0.5,1.5)  {$e_9$};
 \node at (0,-0.4)  {$e_6$};
 \node at (0,-1)  {$e_7$};

 \node at (1,0.4)  {$c_1$};
 \node at (3,0.4)  {$c_2$};

 \node at (-2,0)  {};

\end{tikzpicture}

}
\caption{Smallest pair of non-isomorphic trees with the same values of $\theta_T(\{e_i, e_j\})$}
\end{figure}
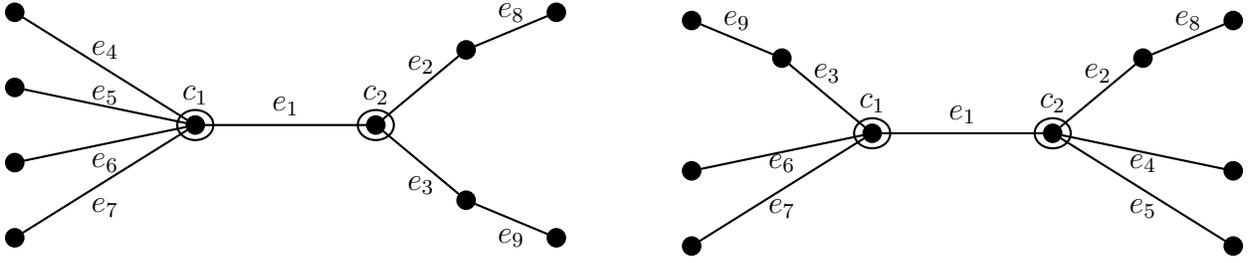

Note that for the left tree $\theta_T(\{e_1,e_2,e_3\})=(5,2,2,1)$ and for the right tree $\theta_T(\{e_1,e_2,e_3\})=(3,3,2,2)$. To fully determine either tree we actually just need the data $\theta_T(\{e_1,e_2,e_3\})$, $\theta_T(\{e_1,e_2,e_4\})$, $\theta_T(\{e_1,e_2,e_5\})$, $\theta_T(\{e_1,e_2,e_6\})$, $\theta_T(\{e_1,e_2,e_7\})$.

Both these examples suggest that our Theorem is essentially optimal. 

$$$$
$$$$
$$$$
We will now show that in general, the data obtained by deleting any $k$ edges determines the data obtained by deleting any $\ell$ edges, for $k \ge \ell$.

\begin{theo}
\label{k->k-1}
Let $T$ be a tree with edge set $\{ e_1, \ldots, e_{n-1}\}$. Then for $k \le n-3$, the set of labeled $k$-cuts determines the set of labeled $(k-1)$-cuts.
\end{theo}
\begin{proof}

More precisely, we claim that knowing $\theta_T(e_{i_1}, \ldots, e_{i_k})$ for all subsets of distinct edges $e_{i_1}, \ldots, e_{i_k}$ determines $\theta_T(e_{i_1}, \ldots, e_{i_{k-1}})$ for any subset of distinct $e_{i_1}, \ldots, e_{i_{k-1}}$.
So assuming we know the former, we will prove that we can determine $\theta_T(e_{i_1}, \ldots, e_{i_{k-1}})$ for arbitrary $i_1, \ldots, i_{k-1}$. 

Consider the range of $\theta_T(e_{i_1}, \ldots, e_{i_{k-1}},x)$ as $x$ runs over all edges in the forest $T'=T-\{e_{i_1},\ldots,e_{i_{k-1}} \}$. We are interested in the edges $\ell$ for which the partition $\theta_T(e_{i_1}, \ldots, e_{i_{k-1}},\ell)$ has the maximum number of ones as components; these edges $\ell$ are exactly the leaf edges of the forest $T-\{e_{i_1},\ldots,e_{i_{k-1}} \}$. It suffices to show that we can use the data $\theta_T(e_{i_1}, \ldots, e_{i_{k-1}},\ell)$ for all such edges $\ell$ to determine $\theta_T(e_{i_1}, \ldots, e_{i_{k-1}})$.

If the partition $\theta_T(e_{i_1}, \ldots, e_{i_{k-1}},\ell)$ is not always the same for each such edge $\ell$, then let $\theta_T(e_{i_1}, \ldots, e_{i_{k-1}},\ell)=(x_{1,\ell}, \ldots, x_{k,\ell},1)$ where $x_{1,\ell} \ge \ldots \ge x_{k,\ell} \ge 1$. We must have

$$\theta_T(e_{i_1}, \ldots, e_{i_{k-1}})=\left(\max_{\ell}x_{1,\ell}, \ldots, \max_{\ell}x_{k,\ell} \right) .$$

If $\theta_T(e_{i_1}, \ldots, e_{i_{k-1}},\ell)$ is the same for each such edge $\ell$ of $T'$, then all trees in the forest $T'$ with more than one vertex have the same number of vertices, say $m$. In other words, $\theta_T(e_{i_1}, \ldots, e_{i_{k-1}})=(m, \ldots, m, 1, \ldots, 1)$ where $m$ occurs some number $M \ge 1$ times and $1$ occurs $k-M$ times. This means that $\theta_T(e_{i_1}, \ldots, e_{i_{k-1}},\ell)=(m, \ldots, m, m-1, 1, \ldots, 1)$, where $m$ occurs $M-1$ times, $m-1$ occurs once, and $1$ occurs $k-M$ times. Clearly $M \cdot m+k-M=n$, so $M(m-1)=M \cdot m-M=n-k \ge 3$ since $k \le n-3$. It follows that $m \ge 4$, $m=3$ and $M \ge 2$, or $m=2$ and $M \ge 3$. These cases exactly correspond to the following three cases for $\theta_T(e_{i_1}, \ldots, e_{i_{k-1}},\ell)$:

\begin{enumerate}
\item $m$ occurs $M-1 \ge 0$ times and $m-1 \ge 3$ occurs exactly once
\item $m = 3$  occurs $M-1 \ge 1$ time and $m-1 = 2$ occurs $1$ time
\item $m=2$ occurs at least $M-1 \ge 2$ times
\end{enumerate}

Based on the number of occurences of $2$ in $\theta_T(e_{i_1}, \ldots, e_{i_{k-1}},\ell)$, we can determine which of the three cases $\theta_T(e_{i_1}, \ldots, e_{i_{k-1}},\ell)$ must be in, and we can thereby determine what $m$ and $M$ must be. Thus we have determined $\theta_T(e_{i_1}, \ldots, e_{i_{k-1}})$.
\end{proof}

Note that the upper bound on $k$ in Theorem \ref{k->k-1} is optimized. If $k=n-2$ then $\theta_T(e_{i_1}, \ldots, e_{i_k})=(2,1, \ldots, 1)$ for any tree $T$, but $\theta_T(e_{i_1}, \ldots, e_{i_{k-1}})$ may vary depending on the tree $T$,  so clearly Theorem \ref{k->k-1} does not hold.  In particular, if $\theta_T(e_{i_1}, \ldots, e_{i_{k-1}},\ell)=(2,1,\ldots,1)$, we can not determine whether $\theta_T(e_{i_1}, \ldots, e_{i_{k-1}})=(3,1,\ldots,1)$ or $(2,2,1, \ldots,1)$. Note that for these two partitions, we have $m=3, M=1$ and $m=2, M=2$, so the proof of Theorem \ref{k->k-1} fails in this case.

On the other hand, for $k=n-1$, Theorem \ref{k->k-1} trivially holds, since  $\theta_T(e_{i_1}, \ldots, e_{i_k})=(1,1, \ldots, 1)$ and $\theta_T(e_{i_1}, \ldots, e_{i_{k-1}})=(2,1, \ldots, 1)$ for any tree $T$.

The following is a corollary of the previous two theorems. It includes all trees, single centroid or double centroid.

\begin{cor}
Let $T$ be a tree on $n \ge 4$ vertices with edge set $\{e_1, \ldots, e_{n-1}\}$ and let $k$ be an integer such that $3 \le k \le n-3$. Then the data $\theta_T(\{e_{i_1}, \ldots, e_{i_k}\})$ for all edges $e_{i_1}, \ldots, e_{i_k}$ uniquely determines $T$.
\end{cor}
\begin{proof}
By repeated application of Theorem \ref{k->k-1}, the data $\{ \theta_T(\{e_{i_1}, \ldots, e_{i_k}\}) \}$ determines the data $\{ \theta_T(e_1, e_2, e_k)\}$ and $\{\theta_T(e_i, e_j)\}$ given in Theorem \ref{2c}. Thus double-centroid trees are determined by the data $\{ \theta_T(\{e_{i_1}, \ldots, e_{i_k}\}) \}$. Of course the data $\{ \theta_T(\{e_{i_1}, \ldots, e_{i_k}\}) \}$ also determines the data $\{\theta_T(e_i, e_j)\}$, so by Theorem 5.7 in Orellana and Scott, we can determine single centroid-trees as well.
\end{proof}

\pagebreak[4]

\section{Computational Methodology}

When computationally verifying Stanley's conjecture, we don't entirely compute $X_G$ for trees on $n$ vertices. Instead, we approach the problem systematically by checking consecutively larger cuts. First we compute the 1-cuts of all trees of a given number of vertices. Then, we collect trees with identical 1-cuts into classes. From there we compute 2-cuts just for those trees with identical 1-cuts and once again collect all those with identical 2-cuts. We repeat this process until all trees are distinguished by their k-cuts, for some k. In our computations, we've found 5-cuts to be sufficient to distinguish all trees up to 24 vertices.
$$$$
The task of searching for trees with identical Chromatic Polynomials was done exhaustively for trees with 1 to 23 vertices. The largest limitation to computation was the nature of growth in number of trees and number of partitions. There is little that can be done to abate the number of unlabelled trees on n vertices, (approximately $3^n$), however a filtration process was implemented to greatly minimize the complexity of comparing all trees on n vertices. 
$$$$
All distinct unlabelled trees on n vertices were generated in a sparse6 graph data format via the free opensource software geng in the gtools utilities provided by nauty \cite{nauty}. In order to efficiently compare trees based on their sets of partitions a preorder on graphs was defined. Sorted cuts and their multiplicity are given lexicographic ordering first in the values of the cuts, then by the values of their multiplicities. A set of cuts can then be sorted according to this ordering, and once sorted, sets of cuts can be compared lexicographically in their cuts. We can now define a preordering on graphs via an ordering on sets of cuts. A graph on n vertices is assigned a string of $n-1$ sets of cuts grouped by number. The first being its sorted set of 1-cuts and the last being the sorted set of $(n-1)$-cuts. Graphs on $n$ vertices are then sorted lexicographically. In other words, if Stanley's conjecture holds this preordering is in fact a total ordering. We have verified that this is the case for $n \leq 24$.
$$$$
Computation time is greatly reduced with this construction because the polynomial is being "lazily evaluated". The trees on n vertices are first be sorted by all $n-1$ of their 1-cuts, and filtered to trees who's orderings match up to this point, only then do we compute all ${n-1}\choose{2}$ 2-cuts, and so on. The source code of our program is available at \cite{files}


\begin{definition}[$\leq \text{on graphs}$]
\begin{itemize}
$$$$
\item Assign Graph G the word $(C_1,...,C_{n-1})$ of its sets of k-cuts, then
\item $G \leq G' \iff (C_1,...C_{n-1}) \leq (C_1',...,C_{n-1}') \text{ lexicographically }$, where $C_i = ( c_1, c_2, ..., c_{\binom{n-1}{i}} )$ $\text{ is a set of integer partitions}$
\item $C_i \leq C_i' \iff (c_1, ..., c_{\binom{n-1}{i}} ) \leq c_1', ..., c_{\binom{n-1}{i}}') \text{ lexicographically, where } c_j = (n_1, ... , n_{i+1})$ $\text{ is an integer partition of n}$ \item $c_j \leq c_j' \iff (n_1, ... ,n_{i+1}) \leq (n_1', ... ,n_{i+1}') \text{ lexicographically}$
\end{itemize}
\end{definition}

\pagebreak[4]

\maketitle

\section{Computational Results and Conjecture}

\*The following tables quantify precisely how many trees have non-unique k-cuts for k=1,2,3.
The second column is the number of trees on the given number of vertices with non-unique k-cuts. The ``density'' column is the second column divided by the total number of trees. It represents what proportion of trees are not distinguished by their k-cuts. The 4th column is the number of distinct families of trees sharing a set of k-cuts. That is, it's the number of equivalence classes, of size larger than 1, of trees with identical k-cuts. 
\*
\begin{multicols}{2}
\begin{flushleft}

\begin{tabular}{|l|l|l|l|}
\hline
vertices&1-cut&density&1-families\\ \hline
7&2&0.182&1\\ \hline
8&6&0.261&3\\ \hline
9&25&0.532&11\\ \hline
10&59&0.557&24\\ \hline
11&178&0.757&60\\ \hline
12&445&0.808&136\\ \hline
13&1154&0.887&274\\ \hline
14&2884&0.913&602\\ \hline
15&7425&0.959&1152\\ \hline
16&18650&0.965&2474\\ \hline
17&47824&0.983&4520\\ \hline
18&122328&0.988&9640\\ \hline
19&316032&0.994&17218\\ \hline
20&819370&0.996&36429\\ \hline
21&2140092&0.998&63813\\ \hline
22&5614634&0.998&135799\\ \hline
23&14817623&0.999&233970\\ \hline
\end{tabular} 

\columnbreak

\begin{tabular}{|l|l|l|l|}
\hline
vertices&2-cut&density&2-families \\ \hline
10&2&0.0189&1\\ \hline
11&2&0.00851&1\\ \hline
12&16&0.029&8\\ \hline
13&20&0.0154&10\\ \hline
14&145&0.0459&72\\ \hline
15&193&0.0249&96\\ \hline
16&1035&0.0536&498\\ \hline
17&1524&0.0313&742\\ \hline
18&7562&0.061&3571\\ \hline
19&11765&0.037&5643\\ \hline
20&54157&0.0658&25210\\ \hline
21&89294&0.0416&42363\\ \hline
22&387625&0.0689&178161\\ \hline
23&1643204&0.111&721593\\ \hline
\end{tabular}

$$$$

\begin{tabular}{|l|l|l|l|}
\hline
vertices&3-cut&density&3-families \\ \hline
16&2&0.000104&1\\ \hline
17&2&0.000041&1\\ \hline
18&22&0.000178&11\\ \hline
19&24&0.000076&12\\ \hline
20&186&0.000226&92\\ \hline
21&216&0.000101&107\\ \hline
22&1348&0.00024&667\\ \hline
23&1554&0.000105&769\\ \hline
\end{tabular}

\end{flushleft}

\end{multicols}  
$$$$
As described above, our algorithm involved taking all trees with $n$ vertices, and checking sequentially larger k-cuts until all trees have been distinguished. As a result, we've found "smallest counterexamples" for k = 1,2,3,4. That is, we have pairs of trees on the least number of vertices such that both trees have identical k-cuts, for k= 1,2,3,4. The trees are given below, in figures 3,4,5,6. Understanding precisely how two trees can have identical k-cuts would be a huge step towards proving Stanley's conjecture.

\includegraphics[width=1\textwidth]{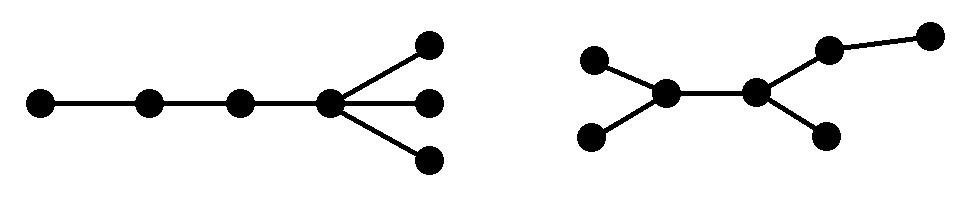}
\\*
\begin{center}Figure 3: The smallest pair of trees with identical 1-cuts.\end{center}
\includegraphics[width=1\textwidth]{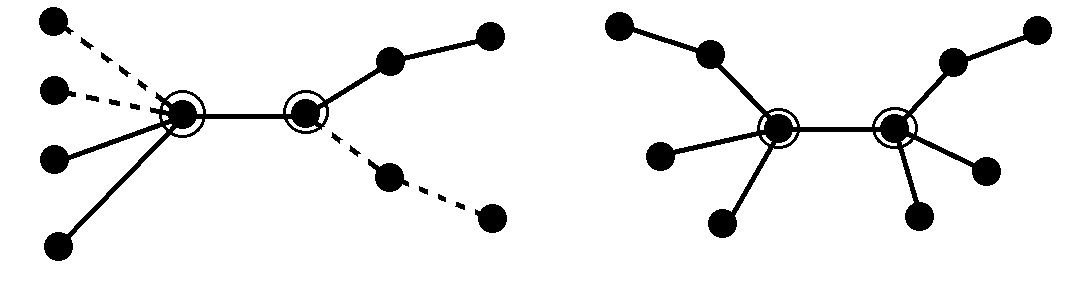}
\\*Figure 4: The smallest pair of trees with identical 2-cuts. Circled vertices are the centroids. Dashed edges on the left tree represent branches that can be swapped across the centroids to obtain the tree on the right.
\\*
$$$$
\includegraphics[width=1\textwidth]{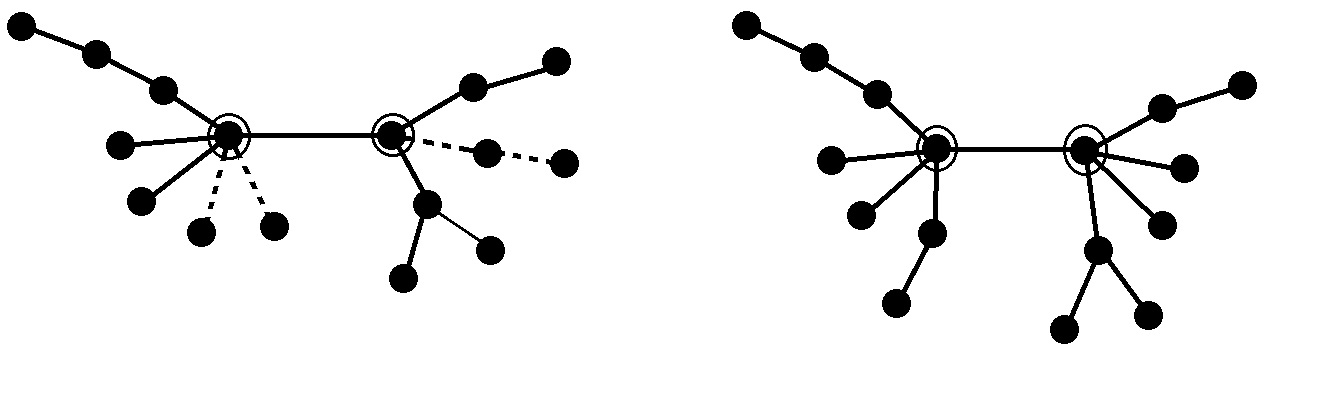}
\begin{center}Figure 5: The smallest pair of trees with identical 3-cuts.\end{center}
$$$$
\includegraphics[width=1\textwidth]{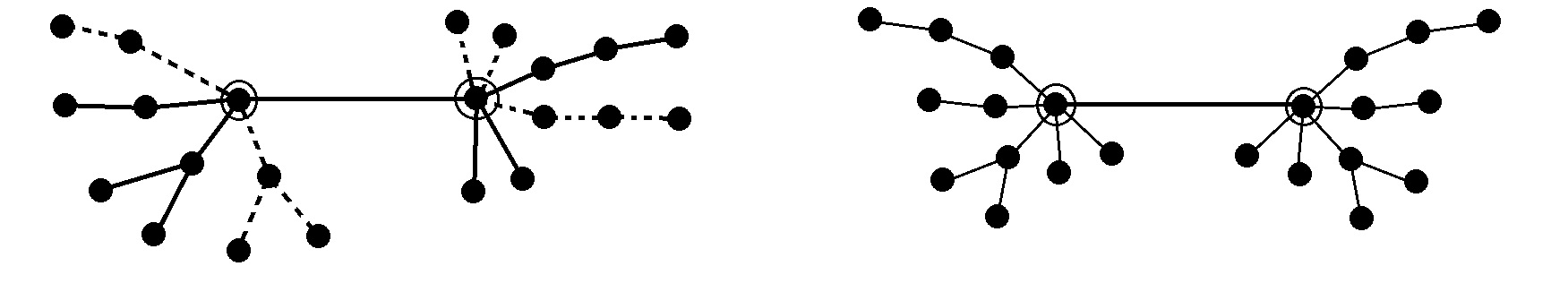}
\begin{center} Figure 6: The smallest pair of trees with identical 4-cuts.\end{center}
\large{
$$$$
\\*
\\*One interesting pattern we see in this data is the ``belayed exponential growth'' of the numbers of trees with non-unique 2(or 3)-cuts. Going from even to odd numbers of vertices, the total number of ``non-uniquely determined by 2-cuts'' trees grows relatively little. When going from an odd to an even order, there's a much larger increase. This is most likely related to the fact that only trees with an even number of vertices can have two centroids, which must have a central edge connecting them. In fig. 4, 5, and 6, one of the trees can be obtained from the other by swapping some of the branches between the centroids. This appears to be the simplest method for which two non-isomorphic trees can have identical k-cuts, which explains why it consistently appears in our minimal examples.
\\*
\\*If one could determine a strict lower bound on the number of vertices required for two trees to have identical k-cuts, this would likely prove Stanely's conjecture. With this in mind, our current efforts are focused on understanding how these ``smallest counterexamples'' given above are generated. 
\\*
\\*Additionally, we note that every minimal pair of trees with identical k-cuts (for k=2,3,4) extends to a pair of n+1 vertex trees with identical k-cuts simply by splitting the central edge and placing a vertex in the middle. Furthermore, these are the only trees on n+1 vertices with identical k-cuts. That is, for both n=22 and n=23, there is precisely one pair of trees with identical 4-cuts, and the pair on 23 vertices is simply the pair on 22 vertices with each given a vertex to split their central edge. This strongly supports the idea that pairs of trees having identical k-cuts is closely related to branch swapping about the centroids. We conjecture that given any pair of two-centroid trees with identical k-cuts, splitting the central edge gives two new trees with identical k-cuts. We have not yet proven this conjecture, but it is supported thus far by our data.

}

\pagebreak[4]

\end{document}